\documentclass{article}
\usepackage{graphicx,amssymb,amsthm,xcolor}      
\def\ind{{\mathchoice {\rm 1\mskip-4mu l} {\rm 1\mskip-4mu l}
{\rm 1\mskip-4.5mu l} {\rm 1\mskip-5mu l}}}

\newtheorem{example}{Example}
\newtheorem{prop}{Proposition}
\newtheorem{remark}{Remark}
\newtheorem{lem}{Lemma}
\begin{document}

\title{ Quantile-based Mean-Field Games\\   with Common Noise} 

\author{Hamidou Tembine  \thanks{This research work is supported by U.S. Air Force Office of Scientific Research under grant number  FA9550-17-1-0259.}
\thanks{The author is with Learning \& Game Theory Laboratory, New York University Abu Dhabi, E-mail:  tembine@nyu.edu}
}

\maketitle

\begin{abstract}                

In this paper, we explore the impact of quantiles on optimal strategies under state dynamics driven by  both individual noise, common noise and Poisson jumps. We first establish an optimality system satisfied the quantile process under jump terms.  We then turn to investigate a new class of finite horizon mean-field games with common noise in which the payoff functional and the state dynamics are dependent not only on the state-action pair but also on  conditional quantiles.  Based on the best-response of the decision-makers, it is shown that the equilibrium conditional quantile process  satisfies a stochastic partial differential equation  in the non-degenerate case. A closed-form expression of the quantile process is provided in a basic Ornstein-Uhlenbeck process with common noise. 
 \end{abstract}



\section{Introduction}

Recent research  has shown the importance of global uncertainty or common noise in mean-field games. This motivates the question of whether the risk associated to the uncertainty can be quantified. 
The quantile is one of the key quantity-of-interests and serves as a baseline for many performance metrics used in risk quantification.  We ask whether the existing tools from mean-field games can be modified to accommodate the quantiles in the payoff functionals and in the dynamics of the states.  We refer to such dynamic interaction as  quantile-based mean-field games. In such games the behavior of the decision maker depends on the state, jump, common noise and quantiles. As the decision affects  the state transitions and the quantile, the quantile in turn influences the individual state. It is  convenient to have the evolution of the quantile for computing optimal strategies of the decision-makers.

\subsubsection*{Related works on mean-field games with common noise }
Anonymous sequential and mean-field games with common noise can be considered as a natural generalization of the mean-field game problems  (see \cite{jova,jova2} and the references therein). 
The works in \cite{bergin1,bergin2}  considered mean-field games with common noise and 
obtained optimality system that determine mean-field equilibria conditioned of the information.
The work in \cite{ref6} provides sufficiency conditions for well-posedness of mean-field games with common noise. Existence of solutions of the resulting stochastic optimality systems are examined in \cite{ref4,ref2,ref3}. A probabilistic approach to the master equation is developed in \cite{ref5}.  In order to determine the optimal strategies of the decision-maker, the previous works used a maximum principle or a master equation which involves a Fokker-Planck equation (see \cite{ref8,ref9,ref7} and the references therein). Due to the inverse nature of the quantile, it is not straightforward to use the master equation proposed in earlier works \cite{ref8,ref9}. The results obtained in \cite{ref6,ref4,ref2,ref3,ref5,ref8,ref9,ref7} for well-posedness, existence and moment estimates do not directly apply to the quantile-based metrics, except under very strong assumptions. The methodology used here extends the work in \cite{kurtz} that is for uncontrolled systems without jump terms.

In this paper we address the question of finding a forward equation satisfied by 
the quantile of  measure-dependent stochastic differential equations with jump terms.  The main difficulty lies in expressing the propagation of the common noise to the distribution of states with jump terms. The presence of the global  (common) noise creates random mean-field quantities: the mean-field limit is a random measure, the conditional mean is random, the cumulative distribution as well as the quantiles are random processes.


\subsubsection*{Contribution }
The contribution of the present work  can be summarized as follows. 
\begin{itemize}\item We introduce a new class of mean-field games with common noise and L\'evy jumps
where the payoff functions and the state dynamics are described based upon conditional quantiles and jumps.
\item We examine  the conditional quantile process as a result of a mean-field  interaction in presence of individual and common noise.  We use It\^o's calculus and implicit function theorem to derive a simple formula for the quantile process.
\item The methodology is illustrated in a simple auction 
mechanism in which the state dynamics mimic a controlled Ornstein-Uhlenbeck process with a common noise. The common noise here captures the  market uncertainty. The decision-makers are prosumers (consumer-producer) who can submit their bids and the respective bids that are below the market price may be selected depending on the quantity needed by the operator to compensate the mismatch between supply and demand in peak hours. By a direct decomposition method, we compute the quantile process when the optimal strategies are employed. We provide a closed-form expression of the quantile process in presence of common noise and show that it satisfies the new dynamics 
established in Proposition 1 and 2.\end{itemize}

\subsubsection*{Organization }
The  rest of the paper is structured as follows.  Section \ref{sec:model} presents a generic game model based on quantiles and with a common noise. Section \ref{sec:wellposdeness} focuses on existence and well-posedness problems.  Section \ref{sec:main} presents the main technical results. In  Section \ref{sec:prosumer} we  provide a basic application to dynamic auction mechanism with common noise.
Section \ref{sec:conclusion} concludes the paper.

Table \ref{tab:TableOfNotationmftg} summarizes  the notations used in the paper.

\begin{table}[htbp]\caption{Table of Notations}
\centering 
\begin{tabular}{r c p{4cm} }
\hline
$\mathcal{I}$ &  $\triangleq$ &set of decision-makers\\
$T$ & $\triangleq$ & Length of the horizon\\
$[0,T]$ & $\triangleq$ & horizon of the mean-field game\\
$t$ & $\triangleq$ & time index.\\
$\mathcal{S}$ & $\triangleq$ & state space\\  
$s_i(t)$ & $\triangleq$ & state of $i$ at time $t$\\ 
$\Delta(\mathcal{S})$ & $\triangleq$ & set of probability measures on $\mathcal{S}$\\   
$m(t,.)$ &  $\triangleq$ &conditional probability measure 
\\ & & of the state at time $t$\\
$f$ &   & fraction in $(0,1)$\\
$m^f(t,.)$ &  $\triangleq$ &conditional quantile process of $s_i(t)$\\
$A$ & $\triangleq$ & action set of  decision-maker $i\in \mathcal{I}$\\  
$a_i(.)$ &  $\triangleq$ &strategy of decision-maker $i\in \mathcal{I}$\\
$\bar{b}(s_i,m^f,a_i)$ &  $\triangleq$ &drift coefficient function\\
$\sigma(s_i,m^f,a_i)$ &  $\triangleq$ &individual diffusion coefficient function\\
$\sigma_o(s_i,m^f,a_i)$ &  $\triangleq$ & global diffusion coefficient function\\
$\gamma(s_i,m^f,a_i,\theta)$ &  $\triangleq$ & individual jump rate coefficient function\\
$r(s_i,m^f,a_i)$ &  $\triangleq$ &instant payoff of decision-maker $i\in \mathcal{I}$\\
$g_i(s_i,m^f)$ &  $\triangleq$ &terminal payoff of decision-maker $i\in \mathcal{I}$\\
$\mathcal{R}_{i,T}(m_0,a_i, m^f)$ &  $\triangleq$ & cumulative payoff of $i$\\
$(p_i, q_i, q_o,\bar{r})$ &  $\triangleq$ & first-order adjoint process of $i\in \mathcal{I}$\\
$H_i$ &  $\triangleq$ &  $r+bp_i+\sigma q_i+{\sigma}_o q_o+\int_{\Theta}\gamma\bar{r}\mu(d\theta)$ of $i\in \mathcal{I}$\\
$H_{i,s}$ &  $\triangleq$ &  $r_{i,s}+\bar{b}_sp_i+\sigma_s q_i+{\sigma}_{o,s} q_o+\int_{\Theta}\gamma_s\bar{r}\mu(d\theta)$ \\ \hline
\end{tabular}
\label{tab:TableOfNotationmftg}
\end{table}

\section{Quantile-based mean-field game} \label{sec:model}
Consider the following mean-field game with common noise. The set of decision-makers  is $\mathcal{I}=\{1,2,\ldots \}.$
 Let $T>0,$  the interval $[0,T]$ is the horizon of the interaction. The state space is $\mathcal{S}=\mathbb{R}.$ The set of probability measures on $\mathcal{S}$ is denoted by $\Delta(\mathcal{S}).$
 For each decision-maker $i,$  a non-empty control action set $A$ is available.  
 An instant payoff of decision-maker $i$  is 
\begin{equation} \label{payoff}
\begin{array}{ll}
r(s_i,m^f,a_i):\ \ \mathcal{S}\times L^1([0,T]\times (0,1),\mathbb{R})\times {A} \rightarrow \mathbb{R},
\end{array}
\end{equation} where $s_i$ is the state of $i,$ $m^f$ is the quantile process,   $a_i$ is the control action of $i$ at time $t$ and $L^1([0,T]\times (0,1),\mathbb{R})$ is the set of integrable processes from $[0,T]\times (0,1)$ to $\mathbb{R}.$
 The dynamics of the state is explicitly given by an  It\^o's stochastic differential equation of Liouville type that is dependent on the quantile, and has a jump and  a common noise term.
\begin{equation} \label{stateone}
\begin{array}{ll}
s_i(t)= s_{i0}+\int_0^t \bar{b} dt'+\int_0^t \sigma dB_i(t')+\int_0^t \sigma_o dB_o(t')\\   \ \ \ +\int_0^t \int_{\Theta} \gamma \tilde{N}_i(dt',d\theta),
\end{array}
\end{equation}
with $s_i(0)=s_{i0}\in \mathcal{S},$ and 
\begin{equation} \label{coefficient}
\begin{array}{ll}
\bar{b},\sigma,\sigma_o(s_i,m^f,a_i):\
 \mathcal{S}\times L^1([0,T]\times (0,1),\mathbb{R})\times A \rightarrow \mathbb{R} ,
\end{array}
\end{equation}
where $\bar{b}$ is the 
drift coefficient functional, $\sigma$ is the diffusion coefficient functional related to the individual noise, $\sigma_o$ is the  diffusion coefficient functional related to the common noise of the game, $\gamma$ is the jump rate. $B_i, B_o$  are  independent  standard Brownian motions on a given probability space $(\Omega, \mathbb{F}, \{\mathcal{F}_t\},\mathbb{P}).$  $B_i, \tilde{N}_i$ represent the local uncertainties  of decision-maker $i$ and $B_o$ is a common noise or global uncertainty. $\mathcal{F}_t$ is the natural  filtration generated by $s_{i0},B_i, N_i, B_o.$ Denote by $\mathcal{F}_t^{B_o}$ the natural filtration generated by the common noise $B_o.$  A solution to the state equation will be denoted by $s_i(t):=s_i^{s_{i0},a_i,m^f}(t).$ The process
$m^f$ is the $f-$quantile associated to the random measure 
$$m(t,ds)=\mathcal{L}(s_{i}(t) \ | \  \mathcal{F}_t^{B_o,m^f})=\mathbb{P}_{s_{i}(t)\ | \ \mathcal{F}^{B_o,m^f}_t}$$ which is the conditional law of the state $s_i(t)$ given  $\mathcal{F}_t^{B_o,m^f}.$
 The quantile process is 
$$m^f(t)=Q_{s_i(t)}(t,f)\,=\,\inf \left\{p\in \mathbb{R}\ :f\leq F(t,p)\right\},$$ where $F(t,p)=m(t,(-\infty, p])$ is the conditional cumulative distribution, this solves $F(t,m^f(t))=f$ even when $F$ has some flat regions or discontinuity.
${N}_i$ is a Poisson
random measure with L\'evy measure $\mu (d\theta),$  independent of $W_i$ and the measure $\mu$ is a $\sigma-$finite measure over $\Theta.$   
$\tilde{N}_i(dt,d\theta)=N_i(dt,d\theta)-\mu(d\theta)dt.$
The jump rate functional  is $${\gamma}: \   \mathcal{S}\times L^1([0,T]\times (0,1),\mathbb{R})\times A \times \Theta \rightarrow \mathbb{R}.$$

 One of the key questions addressed in this paper is to find the dynamics of  the quantile process $m^f.$  Note that, the function $\sigma_o(s_i,m^f,a_i)$ depends on the state, quantile and  the control action. This helps to capture some scenarios where the common noise influences each of  the decision-makers, but  the way it is perceived may be different from one   decision-maker to another.

The quantile mean-field game is played as follows. Given a initial state $s_{i0}$ which is drawn from the initial distribution $m_0,$ the game $\mathcal{G}(s_{i0})$ proceeds as follows. At each instant, each
 decision-maker observes the state (perfect monitoring, perfect state observation), chooses a control action according her strategy 
 (defined below) and observes/measures her payoff.

An admissible control  of decision-maker $i\in \mathcal{I}$ is progressively measurable and integrable process with respect to the filtration $\mathcal{F}_t:=\mathcal{F}_t^{s_0,B_o,B,m^f},$ taking values in $A.$ We denote the set of all admissible control strategy of decision-maker $i$ by $\mathcal{A}:$
$$
\mathcal{A} = \{a_i \in L^1_{\mathbb{F}}(\Omega\times [0,T],\mathbb{R}); a_i(.,t)\in {A}, \ a.e.\  t\in [0,T] \}
$$
where $ L^1_{\mathbb{F}}(\Omega\times [0,T],\mathbb{R})$ is the set of all  progressively $\mathbb{F}-$measurable and integrable $\mathbb{R}-$valued processes defined on $[0,T].$ 
 We identify two processes $a_i$ and $\tilde{a}_i$ in $\mathcal{A}$ if $$ \mathbb{P}(a_i=\tilde{a}_i,\ \ a.e.\  \ on \ [0,T])=1.$$
  A quantile-based strategy for decision-maker $i$ starting at time $0$ is a Borel-measurable map (renamed again by) 
 $$a_i:\  [0, T]\times C([0,T], \mathcal{S})\times  L^1([0,T]\times (0,1),\mathbb{R}) \rightarrow A$$  for which there  exists $\epsilon>0$ such that 
 for any $(t, f_1, f_2)\in (0, T]\times \{   C([0,T], \mathcal{S})\times  L^1([0,T]\times (0,1),\mathbb{R})                  
 \}^2,$ if $f_1=f_2$ on $[0,t]$ then $a_i(., f_1)=a_i(., f_2)$ on $[0, t+\epsilon].$
 To each strategy profile one can associate a control process profile. 
 This will allow us to work with both open-loop and feedback form of strategies by considering $a_i(., s^{s_{i0},a_i}_i, m^f).$
 The cumulative payoff of decision-maker $i$ is $$\mathcal{R}_{i,T}(s_{i0},a_i,m^f):=$$ $$ g(s_i(T),m^f(T))+\int_0^T r(s_i(t),m^f(t),a_i(t))\ dt,$$
 where $g_i(s_i(T),m^f(T)) $ is the terminal payoff of $i.$ The risk-neutral payoff of $i$ is $\mathbb{E}[\mathcal{R}_{i,T}(m_0,a_i,m^f)].$
 The risk-neutral quantile-based mean-field game $\mathcal{G}_{0,T,f}(s_{i0})$ is the normal-form game $\mathcal{G}_{0,T,f}(s_{i0})=(\mathcal{I}, (\mathcal{A},\mathbb{E}\mathcal{R}_{i,T})_{i\in \mathcal{I}} ).$

\subsubsection*{Best response to quantile process}
Each decision-maker $i$ revises her strategy and responds to the stochastic  quantile process $m^f$ and chooses her best response strategy $a_i$ by solving
\begin{equation} \label{br}
\left\{\begin{array}{lll}
\sup_{a_i\in \mathcal{A}} \mathbb{E}[\mathcal{R}_{i,T}(s_{i0},a_i,m^f) ],\\ 
s_i(t):=s_i^{s_{i0},a_i,m^f}(t)\ \mbox{solves}\  (\ref{stateone}),\\
s_i(0)=s_{i0}\sim m_0.
\end{array}
\right.
\end{equation}
A strategy $a_i^*$ such that $\mathbb{E}[\mathcal{R}_{i,T}(s_{i0},a^*_i,m^f) ]=\sup_{a_i\in \mathcal{A}} \mathbb{E}[\mathcal{R}_{i,T}(s_{i0},a_i,m^f) ],$ is called a best response to the quantile process $m^f.$
The set of best response strategies of decision-maker $i$  is denoted by $BR(m^f).$
\subsubsection*{Quantile-based Mean-Field Equilibrium}
Define the following equilibrium point problem:
Find $(a_i^*, m^f)\in \mathcal{A}\times L^1, \ i\in \mathcal{I}$ such that for every decision-maker $i$ one has $a_i^*\in BR(m^f)$ and $m^f$ coincides with $
m^f_*(t)=\,\inf \left\{p\in \mathbb{R}\ :\ m^*(t,(-\infty, p])\geq f\right\},$ where $$m^*(t,ds)=\mathbb{P}_{s^*_{1}(t)\ | \ \mathcal{F}^{B_o,m^f_*}_t}.$$

The difference between the best response problem and the equilibrium problem is that, in the former the conditional quantile process $m^f$ is given, the latter can be seen as a fixed-point as the conditional quantile is the quantile conditioned on the distribution and the common noise, which in turn, is obtained from the states of the individuals.

\section{Well-posedness} \label{sec:wellposdeness}
Well-posedness of the state equation: The state dynamics of decision-maker $i$ is the following controlled Liouville system:

\begin{equation}  \label{ramp2}
\left\{\begin{array}{lll}
ds_i= \bar{b}(s_i, m^f,a_i) dt+ \sigma(s_i, m^f,a_i) dB_i(t)\\ \ \ \quad \quad  +\sigma_0(s_i, m^f,a_i) dB_o(t)+\int_{\Theta}\gamma \tilde{N}_i(dt,d\theta),\\
s_i(0)=s_{i0}\sim m_0,\\
m^f(t)=\,\inf \left\{p\in \mathbb{R}\ :\ m(t,(-\infty, p])\geq f\right\},\\
m(t,ds)=\mathbb{P}_{s^*_{1}(t)\ | \ \mathcal{F}^{B_o,m^f}_t},

\end{array}
\right.
\end{equation}
where $f\in (0,1),$ $m_0$ is the initial measure. 
The probability law of $s_i(t)$ is $\mathbb{E}[\phi(s_i) |  \mathcal{F}_t^{B_o,m} ]=\mathbb{E}[ \phi(s_1(t))\ | \ \mathcal{F}_t^{B_o,m}  ].$

In order to provide a sufficiency condition for existence of  solution to (\ref{ramp2}) we impose the following assumption:

H0: the coefficient functions $\bar{b},\sigma_o,\sigma$ are continuous functions, uniformly Lipschitz in the first component $s,$ twice continuously differentiable in $s_i,$  $\sigma$ non-singular, $s_i(0)$ is $L^1$-integrable and indistinguishable (invariance in distribution per permutation of index).
\begin{lem}[Existence]
Under H0, there exists a positive function of time, $\beta,$ such that $$\sup_{t'\in [0,t] } \mathbb{E} [|s_i(t') |+|m^f(t')|]  < \beta(f, t).$$   There exists at least one  solution to (\ref{ramp2}).
\end{lem}

The proof is by now standard using linear growth recursion, and is therefore omitted. 
 Uniqueness to (\ref{ramp2}) is an open issue. This is because we do not have an estimate bound of the quantile: we are lacking an estimate bound for the quantity $|\bar{b}(s_1,m^f_1, a_i)-\bar{b}(s_2,m^f_2,a_i)|$ as it involves an inverse term.

Stochastic Fokker-Planck Equation: It is important to notice that the conditional law $s^*_{i}(t)\ | \ \mathcal{F}^{B_o,m^f}_t$ is not a deterministic distribution. The random distribution $m(t,ds)$ solves a stochastic Fokker-Planck-Kolmogorov equation of mean-field type.  By It\^o's formula, 
$m(t,s)$ solves the following stochastic partial differential equation (SPDE)  in a weak sense:
For any bounded Borel-measurable function  (test function) $\phi,$
\begin{equation} \label{sfpk}
\left\{\begin{array}{lll}
\int_{s\in \mathbb{R}} \phi(s)m(t,ds)=\int_{s\in \mathbb{R}} \phi(s)m(0,ds)\\
+\int_0^t \int_{s\in \mathbb{R}} \{[\frac{\sigma^2+\sigma^2_o}{2} ]\phi_{ss}+\bar{b}\phi_s \}m(dt',ds)\\
+\int_0^t \int_{s\in \mathbb{R}} \sigma_o \phi_s m(t',ds) dB_o(t')\\ 
+\int_0^t \langle J[\phi],m\rangle dt',
\end{array}
\right.
\end{equation}
where 
$$J[\phi]:=\int_{\Theta} \{ \phi(s(t_{-})+\gamma(t_{-},s(t_{-}),\theta))-\phi(s(t_{-}))-\phi_s.\gamma(t_{-},s(t_{-}),\theta) \ind_{\{ |s| < \epsilon\}} \}\mu(d\theta),$$
and the adjoint $J^*$ of the operator $J$ is given by
$$\langle J[\phi],m\rangle =\langle \phi,J^*[m]\rangle.$$
 \begin{lem}[Conditional density]
The conditional density is stochastic and solves the stochastic integro-PDE given by

\begin{equation} \label{sfpkt1t}
\begin{array}{lll}
dm=\{-(\bar{b}m)_s+ \frac{1}{2}(\sigma^2m)_{ss} +\frac{1}{2}\sigma^2_{o}m_{ss}+J^*[m]\}dt\\     \  \  - \sigma_o m_s  dB_o,
\end{array}
\end{equation}
\end{lem}
\begin{proof}
To establish  the stochastic integro-Fokker-Planck equation, we use the decomposition $s_i(t)=s_{i,nc}(t)+s_{i,c}(t)$ as common part and individual noises part with
$$
s_{i,nc}(t)= \int_0^t \bar{b}dt'+ \int_0^t {\sigma}dB_i(t')+\int_0^t \int_{\Theta} \gamma d\tilde{N}_i(dt',d\theta),
$$

$$
s_{i,c}(t)= s_0 +\int_0^t {\sigma}_o dB_o(t').
$$
The probability law of $s_i(t)\  | \mathcal{F}_t^{B_o,s_0}$ can be expressed in terms of probability of $s_{i,nc}(t):$
$$m_{nc}(s_{nc})=m_{nc}(s-s_c)=m_{nc}(s-s_0 -\int_0^t {\sigma}_o dB_o(t'))$$
We apply It\^o's formula to the latter expression to derive the following:
$$dm=[m_{nc,t}+\frac{ \sigma_o^2}{2} m_{ss}]dt- \sigma_o m_{s} dB_o(t).$$
 The forward equation for the individual noises part $m_{nc}$ solves $$m_{nc,t}=-(\bar{b}m_{nc})_s+ \frac{1}{2}(\sigma^2m_{nc})_{ss} +J^*[m_{nc}].$$
Hence the announced result.
\end{proof}
Solving McKean-Vlasov type SPDE is not a trivial task. Next, we provide a simple approximation of the distribution using virtual agents.
Let the strategy $a_i$ be a progressively measurable and functional of $s_i, m^f.$  Then, the coefficient functions can be rewritten as $$\bar{b}(s_i, m^f,a_i)=\tilde{b}(s_i, m^f) , \sigma(s_i, m^f,a_i)=\tilde{\sigma}(s_i, m^f),  $$  $$\sigma_o(s_i, m^f,a_i)=\tilde{\sigma}_o(s_i, m^f),$$   $$\gamma(s_i, m^f,a_i,\theta)=\tilde{\sigma}_o(s_i, m^f,\theta).$$
The random measure $m$ can be approximated  by an empirical measure process 
$$m(t,ds)=\lim_{n\rightarrow +\infty}\ \ m^n(t,ds)=\lim_{n\rightarrow +\infty}\ \frac{1}{n}\sum_{i\in \mathcal{I}}\delta_{s_{i}(t)}$$

\begin{equation}
\left\{\begin{array}{lll}
ds^n_i=s^n_i(0)+\int_0^t \tilde{b}(s_i^n, m^{f,n}) dt'\\  + \int_0^t  \tilde{\sigma}(s_i^n, m^{f,n}) dB_i(t') +\int_0^t  \tilde{\sigma}_o(s^n_i, m^{f,n}) dB_o(t')\\ +\int_0^t \int_{\Theta}\tilde{\gamma}(s^n_i, m^{f,n},\theta) \tilde{N}_i(dt',d\theta),\\
s^n_i(0)=s^n_{i0}\sim m^n_0,\\
m^{f,n}(t)=\,\inf \left\{s\in \mathbb{R}\ :\ m^n(t,(-\infty, s])\geq f\right\},\\
 m^n(t,.)= \frac{1}{n}\sum_{i\in \mathcal{I}}\delta_{s_{i}(t)},\\
i\in \mathcal{I}.
\end{array}
\right.
\end{equation}


We use an Euler-Maruyama method over the discrete time space $\frac{1}{n}\mathbb{Z} \cap [0,T].$ By de Finetti-Hewitt-Savage's theorem \cite{aldous,sznitman}, it is well-known that there is a random $m$ such that $m^n$ converges in distribution to $m.$ Moreover, conditioning on $m,$ the weak convergence of 
$s^n_i$ is obtained by the convergence of the empirical $m^{f,n}(t).$ These solutions are extended step by step from interval  $[\frac{k}{n},\frac{k+1}{n}]$ to $[\frac{k+1}{n},\frac{k+2}{n}].$
Furthermore, we know that the limiting random measure solves the stochastic Fokker-Planck equation (\ref{sfpk}), it remains to identify an equation satisfied by the quantile process. To do so, we again impose a certain regularity  of coefficient functions to get a positive measure that is continuous with respect to the Lebesgue and that does not vanish on open set.

\section{Main Results} \label{sec:main} 
This section presents our  technical main result.
\subsubsection*{Quantile Dynamics with Jumps}
The next result presents the key quantile dynamics and its relationship with $m$ and  $m_s.$ 
\begin{prop} \label{propq}
Under H0, the quantile process  $m^f$ solves the following stochastic   differential equation:
\begin{equation} \label{ramp2}
\left\{\begin{array}{lll}
m^f(t)=m^f(0_+)+ \int_{0_+}^t  \tilde{b}(m^f,m^f)dt'\\ +\int_{0_+}^t \tilde{\sigma}_o(m^f,m^f)dB_o(t')
-\int_{0_+}^t \tilde{\sigma} \tilde{\sigma}_s(m^f,m^f)dt'\\ -\frac{1}{2}\int_{0_+}^t \tilde{\sigma}^2(m^f,m^f)\frac{m_s(t',m^f)}{m(t',m^f)}dt'\\
-\int_{0_+}^t\int_{\Theta} \tilde{\gamma}(m^f,\theta)\tilde{\gamma}_s(m^f,\theta)\mu(d\theta) dt'\\ -  \frac{1}{2}\int_{0_+}^t \frac{m_s(t',m^f)}{m(t',m^f)}\int_{\Theta}\tilde{\gamma}^2\mu(d\theta) dt'.
\end{array}
\right.
\end{equation}
\end{prop}
The structure of the quantile process was proved in \cite{kurtz} in the controlled-independent and jump-free case ($\gamma=0$). Proposition \ref{propq} extends the work in \cite{kurtz} to include   state-quantile feedback control
 case with jumps under suitable conditions imposed in H0.
\begin{proof} 
In Appendix.
\end{proof}

Since the ratio $\frac{m_s}{m}$ can be written as $(\log m)_s,$ we obtain 
\begin{equation} \label{ramp}
\left\{\begin{array}{lll}
dm^f= \tilde{b}(m^f,m^f)dt\\
-\tilde{\sigma} \tilde{\sigma}_s(m^f,m^f)dt\\ -\frac{1}{2} (\log m)_s \tilde{\sigma}^2(m^f,m^f)    dt     \\
-\int_{\Theta} \tilde{\gamma}(m^f,m^f,\theta)\tilde{\gamma}_s(m^f,m^f,\theta)\mu(d\theta) dt\\ -  \frac{1}{2} 
(\log m)_s \int_{\Theta}\tilde{\gamma}^2(m^f,m^f,\theta)\mu(d\theta) dt\\
+ \tilde{\sigma}_o(m^f,m^f)dB_o(t).
\end{array}
\right.
\end{equation}

\begin{remark}
Noting that $m$ is a random measure given by $\ m(t,ds)=\mathbb{P}_{ s_i(t) \ | \ \mathcal{F}_t^{B_o, m^f}} $ which is  the conditional distribution of $s_i(t)$ given the filtration  $\mathcal{F}_t^{B_o, m^f}=\sigma(B_o(t'), m^f(t'),\ t'\leq t),$ we obtain $m^f$ as a random process satisfying (\ref{ramp}).
When there is no common noise $\sigma_o \rightarrow 0,$ the individual and independent Brownian motions will generate a deterministic quantile if the strategies adopted by the decision-makers are adapted to $\mathcal{F}^{B_i}_t.$  From (\ref{ramp}), the deterministic quantile evolution is given by 

\begin{equation} \label{rampt2}
\left\{\begin{array}{lll}
\dot{m}^f(t)=
 \tilde{b}(m^f,m^f)\\
-\tilde{\sigma} \tilde{\sigma}_s(m^f,m^f) 
\\ -\frac{1}{2} (\log m)_s \tilde{\sigma}^2(m^f,m^f)         \\
-\int_{\Theta} \tilde{\gamma}(m^f,m^f,\theta)\tilde{\gamma}_s(m^f,m^f,\theta)\mu(d\theta)\\ -  \frac{1}{2} 
(\log m)_s \int_{\Theta}\tilde{\gamma}^2(m^f,m^f,\theta)\mu(d\theta).
\end{array}
\right.
\end{equation}
\end{remark}

\begin{example}[Gaussian Quantile]
Let $Z\sim \mathcal{N}(0,1).$
The unit Gaussian $f$-quantile function $Q_Z$ is the function satisfying the condition
$F_{Z}( Q_Z(f)) =f,$
for all $f \in (0,1).$ 
This function $Q_Z(f)$ is not available in simple closed-form, nor the expression of the inverse error function $\mbox{erf}^{-1}(f).$ However, its properties are well understood 
  through its characterization as the solution of the non-linear ordinary differential equation with boundary conditions. 
%
In the case of a Gaussian distribution $s_{nc}(t)= \mu(t)+\sigma(t)Z,$ with mean $\mu(t)$ and variance $\sigma^2(t)$, the quantile function of
that distribution is then $Q_{s_{nc}}(t,f)=\mu(t) + \sigma(t) Q_{Z}(f).$ Clearly, this solves  the system (\ref{ramp}).
  \end{example}
\begin{example}
Consider the hyperbolic tangent function $\mbox{tanh}( s)= \frac{e^{ s}-e^{-  s}}{e^{ s}+e^{- s}}.$ This is clearly a  bounded function and its range is  in $[-1,1].$ Its derivative is 
 $$ \mbox{tanh}_s(s)= \frac{\ (e^{s}+e^{- s})^2-(e^{s}-e^{-s})^2}{(e^{s}+e^{-s})^2}$$ $$= \frac{4 }{(e^{s}+e^{-s})^2}= \frac{1}{(\frac{e^{s}+e^{-s}}{2})^2}=\frac{1}{cosh^2(s)}=1-\mbox{tanh}^2(s),$$ which is again bounded and
  ranges in  $[0,1].$  Let  $\delta>0$ and   the drift $\bar{b}$ to be the hyperbolic tangent function $\delta \mbox{tanh}(\delta s).$ Then, $\bar{b}$  is continuously differentiable, its derivative  $\delta^2[1- \mbox{tanh}^2(\delta s)]$ is bounded with range in $[0,\delta^2].$
 We assume $\sigma,\sigma_o,\gamma(\theta)$ are independent of $s_i, m^f$. Theses functions  are differentiable and their derivative (with respect to $s_i,m^f$) are zero, hence  continuous and bounded. Thus, $(b,\sigma,\sigma_o,\gamma)$ fulfill the conditions in H0. We choose $\gamma(\theta)=\theta.$
 The state dynamics is $$ ds_i=\delta \mbox{tanh}(\delta s_i)dt+\sigma_i dB_i+\sigma_o dB_o+ \int_{\Theta}\theta \tilde{N}(dt,d\theta).
 $$
 with $\Theta=\mathbb{R},\ $
 The operator $J$ is 
 $$J[\phi]:= \int_{\mathbb{R}}\{ \phi(s(t)+\theta)-\phi(s(t)) \}\mu(d\theta),$$
and the adjoint $J^*$ of the operator $J$ is given by

 \begin{equation} \label{rampyy}
\left\{\begin{array}{lll}
\langle J[\phi],m\rangle =\langle \phi,J^*[m]\rangle\\
=\int_{s\in \mathbb{R}} \int_{\theta\in \mathbb{R}}  \{ \phi(s+\theta)-\phi(s) \} \mu(d\theta) m(s) ds\\
=-[\int \mu(\theta)d\theta] \int \phi(s) \ m(s) ds  +\int_{s\in \mathbb{R}}  [\int_{\theta\in \mathbb{R}}   \phi(s+\theta) \mu(\theta)d\theta] m(s) ds\\
=-\lambda \int \phi(s) \ m(s) ds + \int_{s\in \mathbb{R}}  [\int_{y\in \mathbb{R}}   \phi(y) \mu(y-s)dy] m(s) ds\\
=-\lambda \int \phi(s) \ m(s) ds + \int_{y\in \mathbb{R}}    \phi(y)dy  [\int_{s\in \mathbb{R}}  \mu(y-s)m(s) ds] \\
=-\lambda \int \phi(s) \ m(s) ds + \int_{s\in \mathbb{R}}    \phi(s)ds  [\int_{z\in \mathbb{R}}  \mu(-z)m(s+z) dz] \\
=-\lambda \int \phi(s) \ m(s) ds + \int_{s\in \mathbb{R}}    \phi(s)ds  [\int_{\theta\in \mathbb{R}}  \mu(\theta)m(s-\theta) d\theta] \\
= \int \phi(s) ds [ -\lambda m(s)+\int_{\theta\in \mathbb{R}}  \mu(\theta)m(s-\theta) d\theta]
\end{array}
\right.
\end{equation} where $\lambda=\int \mu(\theta)d\theta,\  \mu(d\theta)=\lambda \tilde{\mu}(\theta)d\theta$ and $\tilde{\mu}(\theta),$ is a probability density function.
By identification, 
 \begin{equation} \label{rampyyy}
\left\{\begin{array}{lll}
J^*[m]=-\lambda m(s)+\int_{\theta\in \mathbb{R}}  \mu(\theta)m(s-\theta) d\theta\\
=-\lambda m(s)+\lambda \int_{\theta\in \mathbb{R}}  \tilde{\mu}(\theta)m(s-\theta) d\theta
\end{array}
\right.
\end{equation} 

The stochastic integro-Fokker-Planck equation yields
 \begin{equation} \label{ifpk}
\left\{\begin{array}{lll}
dm(t,s)= \{-(\bar{b}m)_s+ \frac{1}{2}(\sigma^2m)_{ss} +\frac{1}{2}\sigma^2_{o}m_{ss}\}dt\\ 
+[-\lambda m(t,s)+\lambda \int_{\theta\in \mathbb{R}}  \tilde{\mu}(\theta)m(t,s-\theta) d\theta] dt  \\  \  \  - \sigma_o m_s  dB_o,
\end{array}
\right.
\end{equation} 

Set $m(t,s)=m_{nc}(t,s-s_0-\sigma_o B_o(t)).$ Then
 $$m_{nc,t}=-\delta (\mbox{tanh}(\delta s) m_{nc})_s+ \frac{1}{2}(\sigma^2m)_{nc,ss}$$ $$
-\lambda m_{nc}(t,s)+\lambda \int_{\theta\in \mathbb{R}}  \tilde{\mu}(\theta)m_{nc}(t,s-\theta) d\theta$$

Let compute $m_{nc}.$ We try a solution in the following form: $$m_{nc}(t,s)=\frac{1}{2}[e^{\delta s-\frac{1}{2}\delta^2 t}+e^{-\delta s-\frac{1}{2}\delta^2 t}]\tilde{m}_{nc}(t,s).$$
By differentiating with respect to $t,$ one gets
 \begin{equation} \label{ifpk2t}
\left\{\begin{array}{lll}
m_{nc,t}= 
\frac{1}{2}[-\frac{1}{2}\delta^2 e^{\delta s-\frac{1}{2}\delta^2 t}-\frac{1}{2}\delta^2e^{-\delta s-\frac{1}{2}\delta^2 t}]\tilde{m}_{nc} 
+ \frac{1}{2}[e^{\delta s-\frac{1}{2}\delta^2 t}+e^{-\delta s-\frac{1}{2}\delta^2 t}]\tilde{m}_{nc,t}\\
=
-\frac{\delta^2}{4}[ e^{\delta s-\frac{1}{2}\delta^2 t}+e^{-\delta s-\frac{1}{2}\delta^2 t}]\tilde{m}_{nc} 
+ \frac{1}{2}[e^{\delta s-\frac{1}{2}\delta^2 t}+e^{-\delta s-\frac{1}{2}\delta^2 t}]\tilde{m}_{nc,t}
\end{array}
\right.
\end{equation} 

Since 
$$
\left\{\begin{array}{lll}
\frac{1}{2}[e^{\delta s-\frac{1}{2}\delta^2 t}+e^{-\delta s-\frac{1}{2}\delta^2 t}] \delta \mbox{tanh}(\delta s)\\
=\frac{\delta}{2(e^{\delta s}+e^{-\delta s})}[e^{\delta s-\frac{1}{2}\delta^2 t}+e^{-\delta s-\frac{1}{2}\delta^2 t}]  [e^{\delta s}-e^{-\delta s}]\\
=\frac{\delta}{2}[e^{\delta s-\frac{1}{2}\delta^2 t}-e^{-\delta s -\frac{1}{2}\delta^2 t}],
\end{array}
\right.
$$
we obtain
\begin{equation} \label{ifpk22}
\left\{\begin{array}{lll}
\delta (\mbox{tanh}(\delta s) m_{nc})_s = 
\delta (\mbox{tanh}(\delta s) \frac{1}{2}[e^{\delta s-\frac{1}{2}\delta^2 t}+e^{-\delta s-\frac{1}{2}\delta^2 t}]\tilde{m}_{nc})_s\\
= (\frac{\delta}{2}(e^{\delta s-\frac{1}{2}\delta^2 t}-e^{-\delta s -\frac{1}{2}\delta^2 t})\tilde{m}_{nc})_s\\
=\frac{\delta^2}{2}( e^{\delta s-\frac{1}{2}\delta^2 t}+ e^{-\delta s -\frac{1}{2}\delta^2 t})\tilde{m}_{nc}
+\frac{\delta}{2}(e^{\delta s-\frac{1}{2}\delta^2 t}-e^{-\delta s -\frac{1}{2}\delta^2 t})\tilde{m}_{nc,s}
\end{array}
\right.
\end{equation} 

 \begin{equation} \label{ifpk2t}
\left\{\begin{array}{lll}
m_{nc,s}=    \frac{\delta}{2}[e^{\delta s-\frac{1}{2}\delta^2 t}-e^{-\delta s-\frac{1}{2}\delta^2 t}]\tilde{m}_{nc}
+  \frac{1}{2}[e^{\delta s-\frac{1}{2}\delta^2 t}+e^{-\delta s-\frac{1}{2}\delta^2 t}]\tilde{m}_{nc,s},\\
m_{nc,ss}=  \frac{\delta^2}{2}[e^{\delta s-\frac{1}{2}\delta^2 t}+e^{-\delta s-\frac{1}{2}\delta^2 t}]\tilde{m}_{nc}
+\frac{\delta}{2}[e^{\delta s-\frac{1}{2}\delta^2 t}-e^{-\delta s-\frac{1}{2}\delta^2 t}]\tilde{m}_{nc,s}\\
+  \frac{\delta}{2}[e^{\delta s-\frac{1}{2}\delta^2 t}-e^{-\delta s-\frac{1}{2}\delta^2 t}]\tilde{m}_{nc,s}
+\frac{1}{2}[e^{\delta s-\frac{1}{2}\delta^2 t}+e^{-\delta s-\frac{1}{2}\delta^2 t}]\tilde{m}_{nc,ss}\\
=
\frac{\delta^2}{2}[e^{\delta s-\frac{1}{2}\delta^2 t}+e^{-\delta s-\frac{1}{2}\delta^2 t}]\tilde{m}_{nc}
+\frac{2\delta}{2}[e^{\delta s-\frac{1}{2}\delta^2 t}-e^{-\delta s-\frac{1}{2}\delta^2 t}]\tilde{m}_{nc,s}\\
+\frac{1}{2}[e^{\delta s-\frac{1}{2}\delta^2 t}+e^{-\delta s-\frac{1}{2}\delta^2 t}]\tilde{m}_{nc,ss}
\end{array}
\right.
\end{equation}

 \begin{equation} \label{ifpk2tt}
\left\{\begin{array}{lll}
\frac{1}{2}(\sigma^2m_{nc})_{ss}=
\frac{\delta^2\sigma^2}{4}[e^{\delta s-\frac{1}{2}\delta^2 t}+e^{-\delta s-\frac{1}{2}\delta^2 t}]\tilde{m}_{nc}
+\frac{2\delta\sigma^2}{4}[e^{\delta s-\frac{1}{2}\delta^2 t}-e^{-\delta s-\frac{1}{2}\delta^2 t}]\tilde{m}_{nc,s}\\
+\frac{\sigma^2}{4}[e^{\delta s-\frac{1}{2}\delta^2 t}+e^{-\delta s-\frac{1}{2}\delta^2 t}]\tilde{m}_{nc,ss}
\end{array}
\right.
\end{equation}

Putting  together we obtain

 \begin{equation} \label{ifpk3r}
\left\{\begin{array}{lll}
-\frac{\delta^2}{4}[ e^{\delta s-\frac{1}{2}\delta^2 t}+e^{-\delta s-\frac{1}{2}\delta^2 t}]\tilde{m}_{nc} \\
+ \frac{1}{2}[e^{\delta s-\frac{1}{2}\delta^2 t}+e^{-\delta s-\frac{1}{2}\delta^2 t}]\tilde{m}_{nc,t}\\
\\ =
-\frac{\delta^2}{2}( e^{\delta s-\frac{1}{2}\delta^2 t}+ e^{-\delta s -\frac{1}{2}\delta^2 t})\tilde{m}_{nc}\\
-\frac{\delta}{2}(e^{\delta s-\frac{1}{2}\delta^2 t}-e^{-\delta s -\frac{1}{2}\delta^2 t})\tilde{m}_{nc,s}\\
+\frac{\delta^2\sigma^2}{4}[e^{\delta s-\frac{1}{2}\delta^2 t}+e^{-\delta s-\frac{1}{2}\delta^2 t}]\tilde{m}_{nc}\\
+\frac{2\delta\sigma^2}{4}[e^{\delta s-\frac{1}{2}\delta^2 t}-e^{-\delta s-\frac{1}{2}\delta^2 t}]\tilde{m}_{nc,s}\\
+\frac{\sigma^2}{4}[e^{\delta s-\frac{1}{2}\delta^2 t}+e^{-\delta s-\frac{1}{2}\delta^2 t}]\tilde{m}_{nc,ss}
\\
-\frac{\lambda}{2}[e^{\delta s-\frac{1}{2}\delta^2 t}+e^{-\delta s-\frac{1}{2}\delta^2 t}]\tilde{m}_{nc}
\\
+\lambda \int_{\theta\in \mathbb{R}}  \tilde{\mu}(\theta)  \frac{1}{2}[e^{\delta (s-\theta)-\frac{1}{2}\delta^2 t}+e^{-\delta (s-\theta)-\frac{1}{2}\delta^2 t}]\tilde{m}_{nc}(t,s-\theta) d\theta

\end{array}
\right.
\end{equation} 

By choosing $\sigma=1,$ the latter system is reduced to

 \begin{equation} \label{ifpk3s}
\left\{\begin{array}{lll}
+ \frac{1}{2}[e^{\delta s-\frac{1}{2}\delta^2 t}+e^{-\delta s-\frac{1}{2}\delta^2 t}]\tilde{m}_{nc,t}\\
\\ =
+\frac{\sigma^2}{4}[e^{\delta s-\frac{1}{2}\delta^2 t}+e^{-\delta s-\frac{1}{2}\delta^2 t}]\tilde{m}_{nc,ss}
\\
-\frac{\lambda}{2}[e^{\delta s-\frac{1}{2}\delta^2 t}+e^{-\delta s-\frac{1}{2}\delta^2 t}]\tilde{m}_{nc}
\\
+\lambda \int_{\theta\in \mathbb{R}}  \tilde{\mu}(\theta)  \frac{1}{2}[e^{\delta (s-\theta)-\frac{1}{2}\delta^2 t}+e^{-\delta (s-\theta)-\frac{1}{2}\delta^2 t}]\tilde{m}_{nc}(t,s-\theta) d\theta
\end{array}
\right.
\end{equation} 

Thus, the system becomes

 \begin{equation} \label{ifpk3st}
\left\{\begin{array}{lll}
 \tilde{m}_{nc,t}=
\frac{1}{2}\tilde{m}_{nc,ss}
-\lambda\tilde{m}_{nc}
\\
+\frac{2\lambda}{e^{\delta s-\frac{1}{2}\delta^2 t}+e^{-\delta s-\frac{1}{2}\delta^2 t}} \int_{\theta\in \mathbb{R}}  \tilde{\mu}(\theta)  \frac{1}{2}[e^{\delta (s-\theta)-\frac{1}{2}\delta^2 t}+e^{-\delta (s-\theta)-\frac{1}{2}\delta^2 t}]\tilde{m}_{nc}(t,s-\theta) d\theta
\end{array}
\right.
\end{equation} 
By simplifying by $e^{-\frac{1}{2}\delta^2 t},$ we obtain 
\begin{equation} \label{ifpk3str}
\left\{\begin{array}{lll}
 \tilde{m}_{nc,t}=
\frac{1}{2}\tilde{m}_{nc,ss}
-\lambda\tilde{m}_{nc}
\\
+\frac{2\lambda}{e^{\delta s}+e^{-\delta s}} \int_{\theta\in \mathbb{R}}  \tilde{\mu}(\theta)  \frac{1}{2}[e^{\delta (s-\theta)}+e^{-\delta (s-\theta)}]\tilde{m}_{nc}(t,s-\theta) d\theta
\end{array}
\right.
\end{equation} 
We now decompose the term  $\frac{1}{2}[e^{\delta (s-\theta)}+e^{-\delta (s-\theta)}]$ using parallelogram  law: $$
 \frac{1}{2}[e^{\delta (s-\theta)}+e^{-\delta (s-\theta)}]= \frac{1}{2}[e^{\delta s}+e^{-\delta s}]. \frac{1}{2}[e^{-\delta \theta}+e^{\delta \theta}]+ \frac{1}{2}[e^{\delta s}-e^{-\delta s}]. \frac{1}{2}[e^{-\delta \theta}-e^{\delta \theta}]
 $$
 By expanding the terms we arrive at
 \begin{equation} \label{ifpk3str}
\left\{\begin{array}{lll}
 \tilde{m}_{nc,t}=
\frac{1}{2}\tilde{m}_{nc,ss}
-\lambda\tilde{m}_{nc}
\\
+\frac{2\lambda}{e^{\delta s}+e^{-\delta s}} \int_{\theta\in \mathbb{R}}  \tilde{\mu}(\theta)  \frac{1}{2}[e^{\delta s}+e^{-\delta s}]. \frac{1}{2}[e^{-\delta \theta}+e^{\delta \theta}]\tilde{m}_{nc}(t,s-\theta) d\theta\\
+\frac{2\lambda}{e^{\delta s}+e^{-\delta s}} \int_{\theta\in \mathbb{R}}  \tilde{\mu}(\theta)  \frac{1}{2}[e^{\delta s}-e^{-\delta s}]. \frac{1}{2}[e^{-\delta \theta}-e^{\delta \theta}]
\tilde{m}_{nc}(t,s-\theta) d\theta
\end{array}
\right.
\end{equation}

 \begin{equation} \label{ifpk3str2}
\left\{\begin{array}{lll}
 \tilde{m}_{nc,t}=
\frac{1}{2}\tilde{m}_{nc,ss}
-\lambda\tilde{m}_{nc}
\\
+\lambda \int_{\theta\in \mathbb{R}}  \tilde{\mu}(\theta)   \frac{1}{2}[e^{-\delta \theta}+e^{\delta \theta}]\tilde{m}_{nc}(t,s-\theta) d\theta\\
+\lambda \ \mbox{tanh}(\delta s) \int_{\theta\in \mathbb{R}}  \tilde{\mu}(\theta) \frac{1}{2}[e^{-\delta \theta}-e^{\delta \theta}]
\tilde{m}_{nc}(t,s-\theta) d\theta
\end{array}
\right.
\end{equation} 
The term $\int_{\theta\in \mathbb{R}}  \tilde{\mu}(\theta) \frac{1}{2}[e^{-\delta \theta}-e^{\delta \theta}]\tilde{m}_{nc}(t,s-\theta) d\theta$  vanishes. It remains to solve the following system:
 \begin{equation} \label{ifpk3str2}
\left\{\begin{array}{lll}
 \tilde{m}_{nc,t}=
\frac{1}{2}\tilde{m}_{nc,ss}
-\lambda\tilde{m}_{nc}
+\lambda \int_{\theta\in \mathbb{R}}  \tilde{\mu}(\theta)   \frac{1}{2}[e^{-\delta \theta}+e^{\delta \theta}]\tilde{m}_{nc}(t,s-\theta) d\theta
\end{array}
\right.
\end{equation} 
This is latter system is the Fokker-Planck equation associated to the process $ B(t)+\int_{\mathbb{R}} \theta {N}(t,d\theta)$
where  ${N}$ is a Poisson jump process with jump rate $\theta$ and with L\'evy measure $\hat{\mu}(\theta):=\frac{1}{2}[e^{-\delta \theta}+e^{\delta \theta} ]\hat{\mu}(\theta).$
Denote by $\tilde{\mu}_{N}(t,s)$ the probability density of  ${N}(t,.)$ and  by $\tilde{\mu}_{B}(t,s)$ the probability density of  $B(t).$ Then,  the $\tilde{m}_{nc}$ is explicitly given by

$$
\tilde{m}_{nc}(.,s)=[\tilde{\mu}_{B}*\tilde{\mu}_{N}](.,s)=\int_{y}\tilde{\mu}_{B}(t,s-y)\tilde{\mu}_{N}](t,y)dy,$$
 where $*$ denotes the convolution operator. 
$$  m_{nc}(t,s)=\frac{1}{2}[e^{\delta s-\frac{1}{2}\delta^2 t}+e^{-\delta s-\frac{1}{2}\delta^2 t}] [\tilde{\mu}_{B}*\tilde{\mu}_{N}](t,s)
=\frac{1}{2}[e^{\delta s-\frac{1}{2}\delta^2 t}+e^{-\delta s-\frac{1}{2}\delta^2 t}] \int_{y}\tilde{\mu}_{B}(t,s-y)\tilde{\mu}_{N}(t,y)dy
$$
Thus, the closed-form expression is  
$$m(t,s)=m_{nc}(t,s-s_0-\sigma_o B_o(t))= $$ $$
\frac{1}{2}[e^{\delta s-\delta s_0-\delta\sigma_o B_o(t)-\frac{1}{2}\delta^2 t}+e^{-\delta s+\delta s_0+\delta\sigma_o B_o(t)-\frac{1}{2}\delta^2 t}] \int_{y}\tilde{\mu}_{B}(t,s- s_0-\sigma_o B_o(t)-y)\tilde{\mu}_{N}(t,y)dy
$$

The quantile system reduces to
\begin{equation} \label{rampyut}
\left\{\begin{array}{lll}
dm^f= [ \delta tanh(\delta m^f)
 -\frac{1}{2} (\log m)_s   
 -  \frac{1}{2} 
(\log m)_s \int_{\Theta}\theta\mu(d\theta)] dt\\
+ {\sigma}_odB_o(t).
 
\end{array}
\right.
\end{equation}

When $\delta$ vanishes,
the quantile becomes $m^f(t)=s_0+\sigma_oB_o(t)+F^{-1}_{B(t)+\int_{\mathbb{R}} \theta {N}(t,d\theta)}(f),$ where $F^{-1}_X$ is the quantile associated to the mixed process  $X(t)=B(t)+\int_{\mathbb{R}} \theta {N}(t,d\theta).$ It is easily to check that the latter expression solves the quantile SDE (\ref{ramp})..  It is well-known that the Wasserstein distance $W_2$  between two distributions $\mu_0$ and $\mu_1$ is related to the quantile process in one-dimensional case via the relation
$$
W_2^2(\mu_0,\mu_1)=\int_0^1 [ F^{-1}_0(s)-F^{-1}_1(s)]^2 ds= \int_0^1 [ Q_{X_0}(s)-Q_{X_1}(s)(s)]^2 ds= \int_0^1 [ m^f_0-m^f_1]^2 df.
$$
Based on the explicit representation of the quantile one can target best response problems such as  $\inf_{a_i}\ W_2(\mathcal{L}^{target}_i, m_T).$ 
\end{example}


\subsubsection*{Stochastic maximum principle}
Let $(p,q,q_0,\bar{r})$ be the adjoint processes associated to the drift, individual diffusion, common diffusion and jump terms respectively.
Let $H$ be the Pontryagin function $$H(t,s_i,m^f,p,q,q_{o},\bar{r})=r+\bar{b}p+\sigma q+\sigma_oq_o+\int_{\Theta}\gamma \bar{r}(.,\theta)\mu(d\theta) .$$ 
Define the first order adjoint process as
$$
dp=-H_s dt+q dB_i+q_o dB_o+\int_{\Theta} \bar{r}(.,\theta) \tilde{N}_i(dt,d\theta),\ $$  $$ \ p(T)=g_s(s_{i}(T),m^f(T)),
$$
where $H_s$ is a notation for $r_s+\bar{b}_sp+\sigma_s q+\sigma_{o,s}q_o+\int_{\Theta}\gamma_s \bar{r}(.,\theta)\mu(d\theta)$ evaluated at the best response strategy, $$\psi_s(t)=\partial_s\psi(t,s^*_i(t),m^f(t),a_i^*(t)), \ \psi\in \{r,\bar{b},\sigma_o,\sigma\},$$ is the partial derivative with respect to $s.$

\begin{prop}
If the functions $b, \sigma, \sigma_o,r, g$ are  continuously differentiable with respect to $(s_i, m^f),$ all their first-order derivatives with respect to $(s_i,m^f)$ are continuous in $(s_i, m^f, a_i),$ and bounded. Then, the first-order adjoint system is a linear SDE with almost surely bounded coefficient functions. There is a unique $\mathcal{F}^{B_o,B,N,s_0}-$adapted solution such that
$$
\mathbb{E}\left[ \sup_{t\in [0,T]}|p(t)|^2{+}\int_0^T( |q(t)|^2+ |q_o(t)|^2{+}\int_{\Theta}\bar{r}^2\mu(d\theta) )dt\right]  < +\infty.
$$
\end{prop}
\begin{proof}
Assume conditions $H0$  hold. Then,  one obtain a linear SDE  in $(p,q,q_o,\bar{r}).$ In addition, the coefficient of the  linear SDE of Liouiville type has almost surely bounded coefficient function. Such a linear SDE ha a solution. From martingale representation theorem, there is unique solution that is $\mathcal{F}^{B_o,B,N,s_0}-$adapted solution such that
$$
\mathbb{E}\left[ \sup_{t\in [0,T]}|p(t)|^2{+}\int_0^T( |q(t)|^2+ |q_o(t)|^2{+}\int_{\Theta}\bar{r}^2\mu(d\theta) )dt\right]  < +\infty.
$$
 
\end{proof}
These strong smoothness conditions on  $b, \sigma, \sigma_o,r, g$  can be considerably weakened using representations of weak sub/super-differential sets.

\section{Illustrative Example} \label{sec:prosumer}
In this section we examine a basic mean-field game with common noise. The energy sector is evolving with the involvement of prosumers.  A prosumer (producer-consumer) is a user that not only consumes electricity, but can also produce and store electricity. 
We examine how the prosumer energy market can be used to compensate a portion of peak hours energy (NegaWatts). NegaWatts is the amount of demand they shift from peak to off-peak time. The prosumers can contribute to the NegaWatts using their own storage or by selling energy to the operators, which creates a  prosumer market  within the global energy market.
 Figure \ref{market} illustrates a situation in which the prosumer market contribution needs to  be  particularly important between 5pm and 9pm where the production capacity of the energy operators is exceeded. The work in \cite{meyn} proposed individual risk in mean-field control with application to demand response. Their model is a leader-follower mean-field model in which the leader makes decisions based on quantiles. However, their model differs from the common noise   model of mean-field games considered in this paper.
\begin{figure}[htb]
    \centering
    \includegraphics[width=9cm]{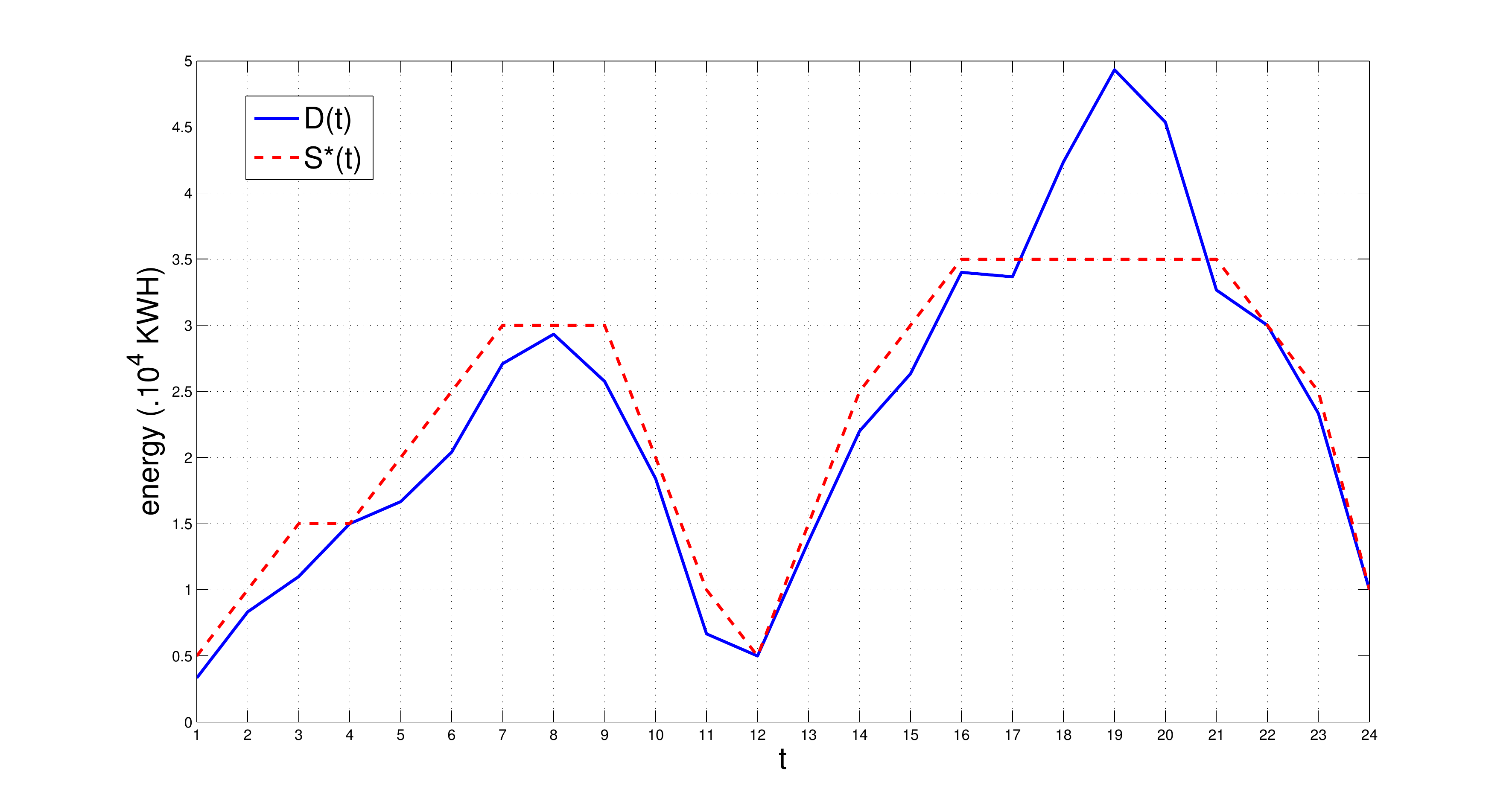}\\ 
    \caption{Prosumer market can be used to compensate the demand during peak hours between 5pm and 9pm.} \label{market}
\end{figure}

\subsubsection*{Multi-Winner Auction }
\begin{figure}[htb]
    \centering
    \includegraphics[width=7cm]{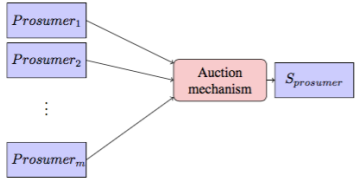}\\ 
    \caption{Multi-Winner Auction between Prosumers} \label{figauction}
\end{figure}

We consider  a dynamic  interaction with $N$ prosumers in an energy market who adopted an auction mechanism  to sell their extra energy to an operator who needs that 
energy during peak hours. We are not imposing the lowest price auction format because there is a total demand and the operator is looking for a certain quantity to compensate 
the NegaWatts. As a consequence, several prosumers may win the auction (Fig. \ref{figauction}). Let us denote the set of prosumers  by $\mathcal{I}=\{1,\ldots, n\}.$ Let us assume that the $\bar{n}$ lowest bids will be selected by the operator, with 
$\bar{n}\leq n.$ The authority allows only positive prices (bids). We change the scale and work with the log-price or log-bid in order to get a wider spectrum that covers all real values.  The log-bid of a prosumer 
belongs to  $(-\infty, +\infty)=\mathbb{R}.$ This will  open a possibility to model the log-bid through unconstrained Brownian motions driven processes and allow us to derive explicit solution. 
\subsubsection*{Supply and Demand Matching}
The interaction between the prosumers and the energy operator starts at time $0.$
At time step $t,$ the energy operator needs a certain quantity to compensate the demand of the customers (for example, during peak hours).
The market clearing condition is when demand equals supply and the bid is less than the market price $p^n(t)$ at time $t.$
The operator is looking for $\bar{n} \bar{Q}(t)$ units of  energy at time $t$.The selected supply of the prosumers at time $t$ is  $$\bar{Q}(t) \sum_{i\in \mathcal{I}} \ind_{\{\log b_{i}(t)\leq \log p^n\}}.$$
We define the mean-field process of the finite population as $$m^n_t=\frac{1}{n}\sum_{i\in \mathcal{I}}\delta_{\log b_{i}(t)}.$$ The demand equals supply if
$\bar{Q}(t) \sum_{i\in \mathcal{I}} \ind_{\{\log b_{i}(t)\leq \log p^n\}}=\bar{n}(t)\bar{Q}(t),$ i.e., $\int_{-\infty}^{\log p^n} m^n(t,db)=\frac{\bar{n}}{n}.$ This means that
$
m^n [t, (-\infty, \log p^n]]=\frac{\bar{n}}{n}.
$

As the quantities are larger, and the scales are chosen such that $\frac{\bar{n}}{n}$ goes to a fraction $f$ of quantity of energy available to the market.
The global market price becomes the $f-$quantile process $m^f$ of the measure $m$ that is the limit of the empirical distribution $m^n$ of the 
log bids. Define the quantile associated to the fraction $f$ at time $t,$ as
$m^f(t)=Q(t,f)\,=\,\inf \left\{p\in \mathbb{R}\ :f\leq F(t,p)\right\},$ where $F(t,p)=m(t,(-\infty, \log p])$ is the cumulative distribution, this solves $F(t,m^f(t))=f,$ even when $F$ has some flat regions or discontinuous. Note that the distribution $m$ will be stochastic due to the global market uncertainty. It allows us to capture a range of prosumer behaviors for the energy of the future. Each prosumer wants to be among the winners of the auction, the ones chosen by the operator and therefore the bids should be below the market price $p$ and the prosumer supply should match the fraction quantity $f.$  The winners gets  a reward  which is $1$ when $a_1\leq a_2\leq \ldots a_{\bar{n}} \leq p^n$ and the total quantity $\bar{n}\bar{Q}(t)$ units that the operator is looking for is achieved.

\subsubsection*{ Log-bid dynamics with common noise}
Here the state of prosumer $i$  at time $t$ is $s_{i}(t),$ which  represents the log bid of prosumer $i.$ $s_i$ is a stochastic process and solves the stochastic 
differential equation (\ref{stateone}) of quantile-type with common noise with the drift function $\bar{b}=\alpha (m^f-s_i)+[a_i-m^f]_+,$ $\sigma_o(t)$ and $\sigma(t)$ are independent of $s_i,m^f.$
 Prosumer $i$ aims to optimize the waste of the remaining energy by selling it in the prosumer market.
The auction ends at time $T.$ The optimal control strategy is trivial in this case and it is given by $a_i^*(t)=m^f(t),$
which is the conditional quantile process. The state dynamics becomes a sort of time-inhomogeneous Ornstein-Uhlenbeck (O-U)  process (see Figure \ref{pros}). 
\begin{figure}[htb]
    \centering
    \includegraphics[width=10cm]{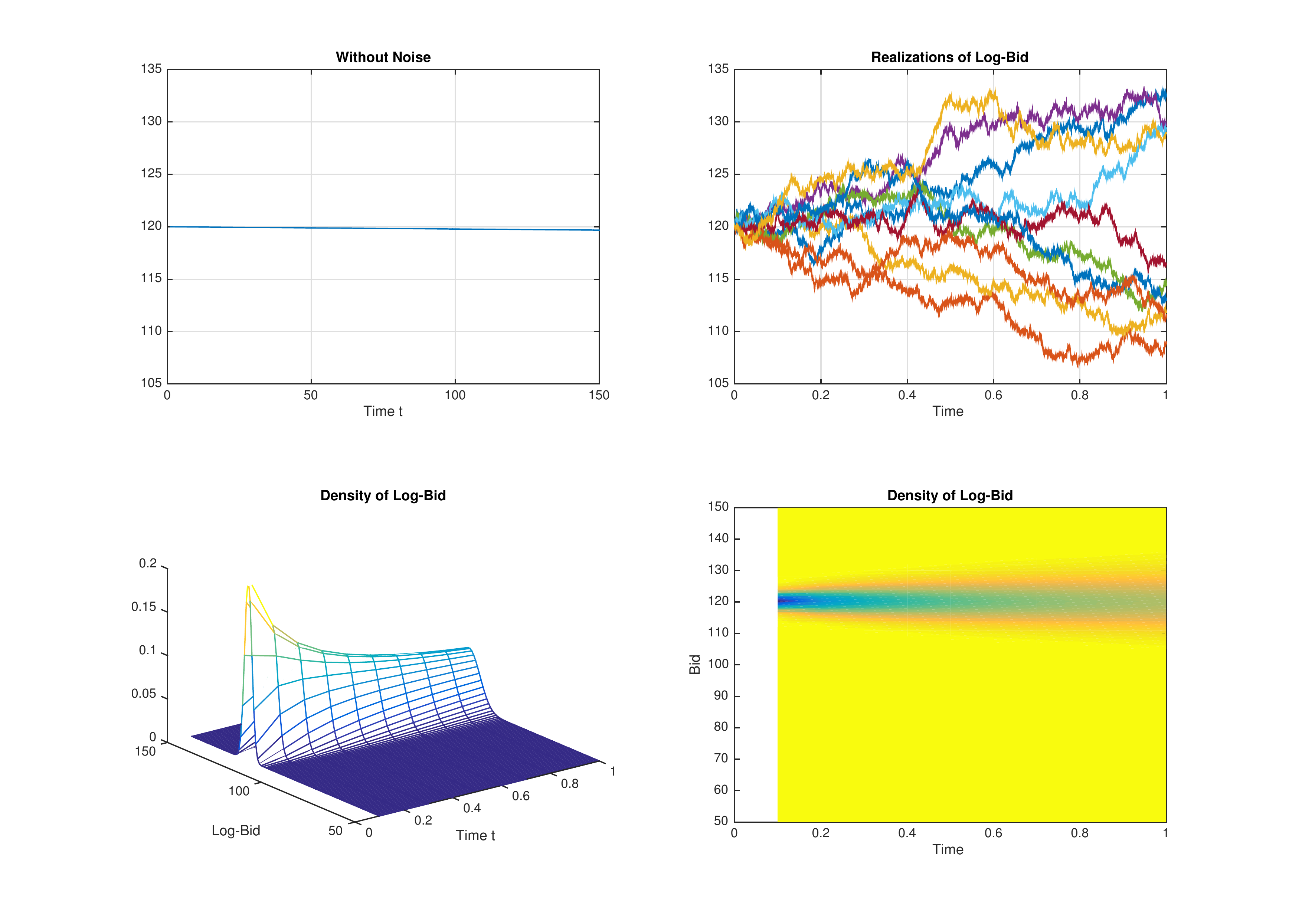}\\ 
    \caption{Samples of Log-bid in the prosumer market.} \label{pros}
\end{figure}
\subsubsection{ Uniqueness of the state }
For any given $L^1$-process $m^f$ the  state dynamics admits a unique solution. Below we provide the solution in closed-form expression.

\subsubsection{ Closed-form expression of the state }
Let $\chi_i$ such that
 $$e^{\alpha t}\chi_i(t)= \alpha \int_0^t e^{\alpha t'}m^f dt'+\int_0^t e^{\alpha t'} \sigma_o dB_o(t')$$
 $$+
 s_{i0}+\int_0^t e^{\alpha t'}\sigma dB_i(t').$$
%

%
%
%
Then, $d\chi_i=-\alpha  \chi_i dt+ \sigma dB_i(t)  +\alpha m^f dt+ \sigma_o dB_o(t).$
It is easy to verify that the above  solution $\chi$  is unique in the sense that if there are two processes $s_i$ and $\chi_i$ solving the state dynamics then $\mathbb{P}(s_i(t)=\chi_i(t),\ a.e. )=1.$

\subsubsection{ Closed-form expression of the optimal strategy }

The equilibrium strategy is the conditional quantile $a_i^*(t)=m^f(t)$ which we explicitly  derive below.
The solution $s_{i}(t)$ is decomposed into individual noise terms and common noise terms as $$s_i(t)=\chi_i(t)=s_{i,c}(t)+s_{i,nc}(t),$$ where
$$
s_{i,c}(t)=e^{-\alpha t}[\alpha \int_0^t e^{\alpha t'}m^f dt'+\int_0^t e^{\alpha t'} \sigma_o dB_o(t')],
$$
 $$ s_{i,nc}(t)=e^{-\alpha t}[ s_{i0}+\int_0^t e^{\alpha t'}\sigma dB_i(t')].$$
 Conditioning on the common noise part we compute
 $$F(s)=P( s_{i}(t)\leq s \ | \ s_{i,c}(t))=$$  $$P(s_{i,nc}(t)\leq s-s_{i,c}(t) | \ s_{i,c}(t))=f.$$
Since $s_{i,nc}(t)$ is a Gaussian process, the inverse of $\Phi(t,.): \ x \mapsto P(s_{i,nc}(t)\leq x)$ is uniquely determined and
$$
s-s_{i,c}(t)=\Phi^{-1}(t,f)=\mu_{i,nc}(t)+\sigma_{nc}(t)Q_Z(f),
$$ with $\mu_{i,nc}(t)=e^{-\alpha t}s_{i0}$ and $\sigma^2_{nc}$ is the variance of $s_{nc}(t).$ 


Thus, $m^f(t)=s_{i,c}(t)+\Phi^{-1}(t,f)=s_{i,c}(t)+m^f_{nc}(t),$ that is, $$
m^f(t)= e^{-\alpha t} [\alpha \int_0^t e^{\alpha t'}m^f dt'+\int_0^t e^{\alpha t'} \sigma_o dB_o(t')]+
m^f_{nc}(t),
$$
where $m^f_{nc}(t)$ is the quantile of the strictly individual noise part, i.e., the quantile of the process $$z_0=e^{-\alpha t}[s_{i0}+\int_0^t e^{\alpha t'}\sigma dB_i(t')].$$  Note that if the initial processes $s_{i0}$ are independent and identically distributed, and independent of $B_i$ then $m^f_{nc}(t)$ the quantile of the process $z_0$ is deterministic.

We compute the deterministic term thanks to Equation (\ref{ramp}). 
\begin{equation}
\left\{\begin{array}{ll}
\frac{m_s(t',m^f)}{m(t',m^f)}=\frac{F_{ss}(t',m^f)}{F_s(t',m^f)}
=\frac{\Phi_{ss}(t',m^f-s_{i,c})}{\Phi_s(t',m^f-s_{i,c})}\\ 
=-\frac{1}{\sigma_{nc}^2}(m^f-s_{i,c}-\mu_{nc})
=-\frac{1}{\sigma_{nc}^2(t)}\sigma_{nc}(t) Q_Z(f)               \\ 
=-\frac{1}{\sigma_{nc}(t)}\ Q_Z(f),    
\end{array}
\right.
\end{equation}
with
\begin{equation}
\left\{\begin{array}{ll} 
\sigma^2_{nc}(t)= var\left[e^{-\alpha t}\int_0^t e^{\alpha t'}\sigma dB_i(t')\right]\\
=\mathbb{E}\left[e^{-\alpha t}\int_0^t e^{\alpha t'}\sigma dB_i(t')\right]^2
=e^{-2\alpha t}\int_0^t e^{2\alpha t'}\sigma^2 dt'.
\end{array}
\right.
\end{equation}
It is easily verified that the variance solves the ordinary differential equation:
$
\alpha \sigma_{nc}+\dot{\sigma}_{nc}=\frac{\sigma^2}{2\sigma_{nc}}.
$
By Proposition 1 we check that
\begin{equation}
\left\{\begin{array}{ll}
dm^f=\sigma_o dB_o-\frac{\sigma^2}{2}\frac{m_s(t,m^f)}{m(t,m^f)} dt\\
= \sigma_o dB_o +\frac{\sigma^2}{2\sigma_{nc}(t)}\   Q_Z(f)dt,\\
= \sigma_o dB_o +(\alpha \sigma_{nc}+\dot{\sigma}_{nc})  Q_Z(f) dt.
\end{array}
\right.
\end{equation}

Using the latter equation we compute the operator cost during peak hours.
The mismatch between the supply $S_o$ of the operator and the demand  $D$ of the consumers creates  the certain cost $$c(T,D(T),S_o(T))=\mathbb{E}[m^f(T) (D(T)-S_o(T))_+]$$ which is obtained from the stochastic differential equation solved by the optimal quantile process.
\section{Conclusion} \label{sec:conclusion}
In this paper we have introduced a new class of mean-field games called quantile-based mean-field games. Due to the inverse nature of the quantile, the existing methodologies previously developed for existence, uniqueness and moment estimates fail, except under very strong assumptions. In this work, a weak solution approach was proposed and a direct method was employed as an illustration. Explicit closed-form expressions are provided a class of mean-field games with $tanh(.)$ as a drift function. However, number of issues remain to be addressed. The term $(\log m)_s$ plays a crucial role in the dynamics is not available in the current formulation. In our future work, we plan to investigate possible  representation of the log-likelihood in terms $m^f$ without using directly $m, m_s.$
\section*{Acknowledgment}
This research work is supported by U.S. Air Force Office of Scientific Research under grant number  FA9550-17-1-0259


%
\section*{Proof of the Quantile Dynamics, Proposition  \ref{propq}}

Consider the implicit equation with stochastic quantities $q$ and $m:$
$$F(q, m)=P(s_i\leq q)-f=\int_{-\infty}^q m(s)ds-f=0.$$
Assuming regularity we have the stochastic integro-PDE solved by $m | B_o:$

\begin{equation} \label{sfpkt1}
\left\{\begin{array}{lll}
dm=\{-(\bar{b}m)_s+ \frac{1}{2}(\sigma^2m)_{ss} +\frac{1}{2}\sigma^2_{o}m_{ss}+J^*[m]\}dt
\\  \ \quad \quad - \sigma_o m_s  dB_o,
\end{array}
\right.
\end{equation}
 It\^o's formula applied to the implicit equation becomes
\begin{equation}
\left\{\begin{array}{ll}
0=dF\\
=F_t dt+F_q dq+F_m. dm \\
+\frac{1}{2}[ F_{qq} d\langle q \rangle_t + (F_{qm}+F_{mq}) d\langle q,m \rangle_t+F_{mm}  d\langle  m \rangle_t],
\end{array}
\right.
\end{equation}

Note that  $F_q=m(t,q)$ which is stochastic quantity. 

The main difficulty here is a theory of differentiation with respect to $m.$ We can use different notions such as Wasserstein gradient, Fr\'echet derivative, functional derivative, G\^ateaux  derivative. However, here, the derivative we will need involve $m(t,q)$ which is the conditional density evaluated to the quantile, which is obviously a degenerate quantity. 
In order the make the differentiation part self-content, we  use a mollifier approximation of the conditional distribution of $s_i(t)\  | \ \mathcal{F}_t^{B_o,m}.$
$$
m(t,dy)=\lim_{(n,\epsilon) \rightarrow (+\infty, 0)}\ \frac{1}{n}\sum_{j=1}^n \rho_{\epsilon}(y-s_i),
$$
where $\rho_{\epsilon}(x)=\frac{1}{\epsilon^d}\rho(\frac{x}{\epsilon}),$\  $\int_{x\in \mathbb{R}} \rho(x) dx=1.$

The implicit equation, by renaming $F[{n,\epsilon}]=:F$ becomes

$$F(q, s)=\int_{-\infty}^q \frac{1}{n}\sum_{j=1}^n \rho_{\epsilon}(y-s_i)dy -f=0.$$

Note that  $F_q(q,s_1,\ldots, s_n)=\frac{1}{n}\sum_{j=1}^n \rho_{\epsilon}(q-s_i)\neq 0$ which is non-zero stochastic quantity. Since $F_q\neq 0,$ one can use the implicit function theorem and it provides the existence of a local function $q=q(f,s_1,\ldots, s_n).$  
\begin{equation}
\left\{\begin{array}{ll}
0=dF\\
=F_t dt+F_q dq+F_s. ds \\
+\frac{1}{2}\{ F_{qq} d\langle q \rangle_t \\ + \sum_{i=1}^n (F_{qs_i}+F_{s_iq}) d\langle q,s_i \rangle_t+F_{s_is_i}  d\langle  s_i \rangle_t\}\\
\end{array}
\right.
\end{equation}
We compute the first and second derivative terms that are needed in the It\^o's formula:
\begin{equation}
\left\{\begin{array}{ll}
F_q=\frac{1}{n}\sum_{i=1}^n \rho_{\epsilon}(q-s_i),\\
F_{qq}= \frac{1}{n}\sum_{i=1}^n \rho'_{\epsilon}(q-s_i),\\
F_{s_iq}=- \frac{1}{n} \rho'_{\epsilon}(q-s_i)=F_{qs_i},\\
F_{s_i}=\frac{1}{n} \int_{-\infty}^q \rho'_{\epsilon}(y-s_i)dy=-\frac{1}{n}\rho_{\epsilon}(q-s_i),\\
F_{s_is_i}=  \frac{1}{n} \int_{-\infty}^q \rho''_{\epsilon}(y-s_i)dy=\frac{1}{n}\rho'_{\epsilon}(q-s_i),
\end{array}
\right.
\end{equation}

Using the forward SDE solved by $s$ we compute the following terms:
\begin{equation}
\left\{\begin{array}{ll}
F_s. ds=\sum_{i=1}^n F_{s_i} ds_i\\
=\sum_{i=1}^n F_{s_i} [ \bar{b}dt+\sigma dB_i+\sigma_o dB_o+\int_{\Theta}{\gamma} \tilde{N}_i(dt,d\theta)]\\
= -\sum_{i=1}^n \frac{1}{n}\rho_{\epsilon}(q-s_i)[ \bar{b}dt+\sigma dB_i+\sigma_o dB_o+\int_{\Theta}{\gamma} \tilde{N}_i(dt,d\theta)]
\end{array}
\right.
\end{equation}

It follows that
\begin{equation}
\left\{\begin{array}{ll}
-F_q dq=\sum_{i=1}^n F_{s_i} [ \bar{b}dt+\sigma dB_i+\sigma_o dB_o+\int_{\Theta}{\gamma} \tilde{N}_i(dt,d\theta)] \\
+\frac{1}{2}\{ F_{qq} d\langle q \rangle_t \\ + \sum_{i=1}^n (F_{qs_i}+F_{s_iq}) d\langle q,s_i \rangle_t+F_{s_is_i}  d\langle  s_i \rangle_t\}\\
\end{array}
\right.
\end{equation}
By identification of the diffusion terms we obtain

\begin{equation}
\left\{\begin{array}{ll}
dq=- \sum_{i=1}^n \left(  \frac{\sigma}{F_q} F_{s_i}\right)dB_i-\left( \frac{ \sum_{i=1}^n \sigma_oF_{s_i}}{F_q}\right) dB_o\\
- \left(\frac{1}{F_q}\sum_{i=1}^n \bar{b} F_{s_i} \right) dt\\
- \left(\frac{1}{F_q}\sum_{i=1}^n \int_{\Theta}{\gamma} \tilde{N}_i(dt,d\theta)F_{s_i} \right) \\
-\frac{1}{2F_q}\{ F_{qq} d\langle q \rangle_t + \\ \sum_{i=1}^n 2F_{qs_i} d\langle q,s_i \rangle_t+F_{s_is_i}  d\langle  s_i \rangle_t\}
\end{array}
\right.
\end{equation}

Using the relation $\sum_{i} F_{s_i}=-F_q,$ and taking the limit we obtain

\begin{equation} \label{ertyu}
\left\{\begin{array}{ll}
dq=\bar{b}(q,q,a(q,.)) dt+\sigma_o(q,q,a(q,.)) dB_o\\
- \lim\ \sum_{i=1}^n \left(  \frac{\sigma}{F_q} F_{s_i}\right)dB_i\\
-  \lim\ \sum_{i=1}^n \int_{\Theta}\frac{\gamma F_{s_i}}{F_q} \tilde{N}_i(dt,d\theta) \\
- \lim\  \frac{1}{2F_q}\{ F_{qq} d\langle q \rangle_t + \\ \sum_{i=1}^n 2F_{qs_i} d\langle q,s_i \rangle_t+F_{s_is_i}  d\langle  s_i \rangle_t\}
\end{array}
\right.
\end{equation}
We have
$\lim_{n,\epsilon} F_q=\lim \frac{1}{n}\sum_{j=1}^n \rho_{\epsilon}(q-s_i)=\lim \int \rho_{\epsilon}(q-s)m(x)dx=m(q).$

We now compute the quadratic variation terms

\begin{equation}
\left\{\begin{array}{ll}
d\langle q \rangle_t =[(\frac{\sum_j \sigma_o(s_j,.) F_{s_j}}{F_q})^2+ \sum_{i=1}^n \left(  \frac{\sigma(s_i,.)}{F_q} F_{s_i}\right)^2] dt\\
+ \sum_{i=1}^n \left(  \int_{\Theta}\frac{\gamma(s_i,.)\mu(d\theta)}{F_q} F_{s_i}\right)^2] dt
,\\
d\langle q,s_i \rangle_t=[ \sigma_o(s_i,.)\left( \frac{ \sum_{j=1}^n \sigma_o(s_j,.)F_{s_j}}{F_q}\right)-  \frac{\sigma^2}{F_q} F_{s_i}]dt\\
 -  [\int_{\Theta}\frac{\gamma^2 \mu(d\theta)}{F_q} F_{s_i}]dt,
\\
d\langle  s_i \rangle_t=[\sigma_o^2(s_i,.)+\sigma^2(s_i,.)+\int_{\Theta}\gamma^2(s_i,.)\mu(d\theta)]dt.
\end{array}
\right.
\end{equation}
The limit of the second derivative $F_{qq}$ is
$\lim F_{qq}=\lim \frac{1}{n}\sum_{j=1}^n \rho'_{\epsilon}(q-s_i)=\lim \int \rho'_{\epsilon}(q-s)m(s)ds=\lim \int \rho_{\epsilon}(q-s)m_s(s)ds,$
where we have used an integration by part.
\begin{equation}
\left\{\begin{array}{ll}
-\frac{F_{qq}}{2F_q}d\langle q \rangle_t=  - \frac{F_{qq}}{2F_q}\{  (\sum_{j=1}^n \frac{\sigma_oF_{s_j}}{F_q})^2+\sum_{i}\sigma^2 \frac{F_{s_i}^2}{F_q^2} \}\\
\quad \quad - \frac{F_{qq}}{2F_q}\sum_{i}[\int_{\Theta}\frac{\gamma^2F_{s_i}}{F_q}\mu(d\theta)]^2\\
\mbox{which goes to: } \\
-\frac{m_s(q)}{2m(q)} \sigma_o^2 (q,.) 
\end{array}
\right.
\end{equation}
because  $$\lim\ \sum_{i}\sigma^2 F_{s_i}^2 =0,$$   $$ \lim_{n,\epsilon}\ \sum_{i}[\int_{\Theta}\frac{\gamma^2F_{s_i}}{F_q}\mu(d\theta)]^2=0.$$

\begin{equation}
\left\{\begin{array}{ll}
-\sum_i\frac{F_{qs_i}}{F_q}d\langle q, s_i \rangle_t=  - (\sum_i \frac{\sigma_o F_{s_i}}{F_q})(  \sum_j \sigma_o\frac{F_{qs_j}}{F_q})\\
+\sum_{i=1}^n \sigma^2\frac{F_{qs_i}}{F_q}.\frac{F_{s_i}}{F_q}\\
+ \sum_{i=1}^n \frac{F_{qs_i}}{F_q}.\ [\int_{\Theta}\frac{\gamma^2 \mu(d\theta)}{F_q} F_{s_i}]
\\
\mbox{which goes to: } \\
\frac{\sigma_o(q,.)}{m(q)} (\sigma_om)_s(q) ,
\end{array}
\right.
\end{equation}
as the terms $\sum_{i=1}^n \sigma^2\frac{F_{qs_i}}{F_q}.\frac{F_{s_i}}{F_q}$  and  $\sum_{i=1}^n \frac{F_{qs_i}}{F_q}.\ [\int_{\Theta}\frac{\gamma^2 \mu(d\theta)}{F_q} F_{s_i}]$ vanish.

\begin{equation}
\left\{\begin{array}{ll}
- \sum_i\frac{F_{s_is_i}}{2F_q}d\langle s_i \rangle_t=  - \sum_i\frac{F_{s_is_i} (\sigma_o^2+\sigma^2+\int_{\Theta}\gamma^2(s_i,.)\mu(d\theta))}{2F_q} \\
\mbox{which goes to: } \\
-\frac{1}{2m(q)} [m (\sigma_o^2+\sigma^2+\int_{\Theta}\gamma^2(s,.)\mu(d\theta))]_s(q)
\end{array}
\right.
\end{equation}
Summing the latter terms, we obtain

\begin{equation}
\left\{\begin{array}{ll}
-\lim_{n,\epsilon}\ \frac{F_{qq}}{2F_q}d\langle q \rangle_t
-\lim_{n,\epsilon}\ \sum_i\frac{F_{qs_i}}{F_q}d\langle q, s_i \rangle_t
-\lim_{n,\epsilon}\ \sum_i\frac{F_{s_is_i}}{2F_q}d\langle s_i \rangle_t\\
= -\frac{1}{2m(q)}\{  \sigma_o^2 (q,.)m_s(q)-2 \sigma_o(q,.)  (\sigma_om)_s(q)\\ +  (m \sigma_o^2)_s+ (m \sigma^2)_s +[m\int_{\Theta}\gamma^2(s,.)\mu(d\theta))]_s\}dt\\
=-\frac{1}{2m(q)}\{ \sigma_o^2m_s-2 \sigma_o\sigma_{o,s} m-2 \sigma_o^2m_s\\ +(m \sigma_o^2)_s+ (m \sigma^2)_s +[m\int_{\Theta}\gamma^2(s,.)\mu(d\theta))]_s\}dt\\
=-\frac{1}{2m(q)}\{ -\sigma_o^2m_s-2 \sigma_o\sigma_{o,s} m +(m \sigma_o^2)_s+[m\int_{\Theta}\gamma^2(s,.)\mu(d\theta))]_s\}dt\\
-\frac{1}{2m(q)} (m \sigma^2)_s dt\\
=-\frac{1}{2m(q)} (m\sigma^2)_s (q)dt 
-\frac{1}{2m(q)}[m\int_{\Theta}\gamma^2(s,.)\mu(d\theta))]_s dt
\end{array}
\right.
\end{equation}

Equation (\ref{ertyu}) becomes
\begin{equation}
\left\{\begin{array}{ll}
dq=\bar{b}(q,q,a(q,.)) dt+\sigma_o(q,q,a(q,.)) dB_o\\
-\frac{1}{2m(q)} (m \sigma^2)_s(q) dt -\frac{1}{2m(q)}[m\int_{\Theta}\gamma^2\mu(d\theta))]_s dt\\
=\bar{b}(q,q,a(q,.)) dt+\sigma_o(q,q,a(q,.)) dB_o\\
-\sigma(q,.)\sigma_s(q,.) dt- \frac{\sigma^2}{2} \frac{m_s(q)}{m(q)} dt \\
-\int \gamma(q,.)\gamma_s(q,.)\mu(d\theta) dt- \frac{\int_{\Theta}\gamma^2\mu(d\theta)}{2} \frac{m_s(q)}{m(q)} dt 

\end{array}
\right.
\end{equation}
This completes the proof.

\section*{Author Information}

{\it Hamidou Tembine} (S'06-M'10-SM'13) received the M.S. degree in Applied Mathematics from Ecole Polytechnique in 2006 and the Ph.D. degree in Computer Science  from University of Avignon in 2009. His current research interests include evolutionary games, mean field stochastic games and applications.
In December 2014, Tembine received the IEEE ComSoc Outstanding Young Researcher Award for his promising research activities for the benefit of the society. He was the recipient of 7 best article awards in the applications of game theory. Tembine is a prolific researcher and holds several  scientific publications including magazines, letters, journals and conferences. He is author of the book on "distributed strategic learning for engineers" (published by CRC Press, Taylor \& Francis 2012), and co-author of the book "Game Theory and Learning in Wireless Networks" (Elsevier Academic Press). Tembine has been co-organizer of several scientific meetings on game theory in networking, wireless communications and smart energy systems.  He is recipient of the Next Einstein Fellow award  2017.  He is a senior member of IEEE.


\end{document}